\documentclass[11pt]{amsart}

\usepackage{latexsym}
\usepackage{amssymb,amsfonts,amsmath,amsthm}
\usepackage{graphicx}
\usepackage{enumerate}
\usepackage{mdwlist}
\usepackage[margin=1.3in]{geometry}

\newtheorem{theorem}{Theorem}  

\newtheorem{lemma}[theorem]{Lemma}

\newtheorem{proposition}[theorem]{Proposition}
\newtheorem{conjecture}[theorem]{Conjecture}
\newtheorem{definition}[theorem]{Definition}

\def\COMMENT#1{}
\def\TASK#1{}

\numberwithin{theorem}{section}
\numberwithin{equation}{section}

\newdimen\margin   
\def\textno#1&#2\par{%
   \margin=\hsize
   \advance\margin by -4\parindent
          \setbox1=\hbox{\sl#1}%
   \ifdim\wd1 < \margin
      $$\box1\eqno#2$$%
   \else
      \bigbreak
      \hbox to \hsize{\indent$\vcenter{\advance\hsize by -3\parindent
      \it\noindent#1}\hfil#2$}%
      \bigbreak
   \fi}
\def\noproof{{\unskip\nobreak\hfill\penalty50\hskip2em\hbox{}\nobreak\hfill%
       $\square$\parfillskip=0pt\finalhyphendemerits=0\par}\goodbreak}
\def\endproof{\noproof\bigskip}

\title{Fractional and integer matchings in uniform hypergraphs}

\author{Daniela K\"uhn, Deryk Osthus and Timothy Townsend}
\thanks {The research leading to these results was partially supported by the  European Research Council
under the European Union's Seventh Framework Programme (FP/2007--2013) / ERC Grant
Agreement no. 258345 (D.~K\"uhn) and 306349 (D.~Osthus).}

\begin{document}

\begin{abstract}
Our main result improves bounds of Markstr\"om and Ruci\'nski on the minimum $d$-degree which forces a perfect matching in a $k$-uniform hypergraph on $n$ vertices. We also extend bounds of Bollob\'as, Daykin and Erd\H{o}s by asymptotically determining the minimum vertex degree which forces a matching of size $t< n/2(k-1)$ in a $k$-uniform hypergraph on $n$ vertices. Further asymptotically tight results on $d$-degrees which force large matchings are also obtained. Our approach is to prove fractional versions of the above results and then translate these into integer versions.
\end{abstract}

\date{\today}

\maketitle 

\section{Introduction}\label{Introduction}
\subsection{Large matchings in hypergraphs with large degrees}\label{subsection 1.1}

A $k$-\textit{uniform hypergraph} is a pair $G=(V,E)$ where $V$ is a finite set of vertices and the edge set $E$ consists of unordered $k$-tuples of elements of $V$. A \textit{matching} (or \textit{integer matching}) $M$ in $G$ is a set of disjoint edges of $G$. The \textit{size} of $M$ is the number of edges in $M$. We say $M$ is \textit{perfect} if it has size $|V|/k$. Given $S\in \binom{V}{d}$, where $0\leq d\leq k-1$, let $deg_G(S)=|\{e\in E : S\subseteq e\}|$ be the \textit{degree} of $S$ in $G$. Let $\delta_d(G)=\min_{S\in \binom{V}{d}}\{deg_G(S)\}$ be the \textit{minimum d-degree} of $G$. When $d=1$, we refer to $\delta_1(G)$ as the \textit{minimum vertex degree} of $G$. Note that $\delta_0(G)=|E|$.

For integers $n$, $k$, $d$, $s$ satisfying $0\leq d\leq k-1$ and $0\leq s\leq n/k$, we let $m_d^s(k,n)$ denote the minimum integer $m$ such that every $k$-uniform hypergraph $G$ on $n$ vertices with $\delta_d(G)\geq m$ has a matching of size $s$. We write $o(1)$ to denote some function that tends to $0$ as $n$ tends to infinity. The following degree condition for forcing perfect matchings has been conjectured in \cite{Absorbing, KO} and has received much attention recently.
\begin{conjecture}\label{main conjecture}
Let $n$ and $1\leq d\leq k-1$ be such that $n$, $d$, $k$, $n/k\in \mathbb{N}$. Then
$$m_d^{n/k}(k,n)= \left( \max\left\{\frac{1}{2}, 1-\left(\frac{k-1}{k}\right)^{k-d}\right\}+o(1)\right)\binom{n-d}{k-d}.$$ 
\end{conjecture}
The first term in the lower bound here is given by the following parity-based construction from \cite{KOtwo}. For any integers $n$, $k$, let $H'$ be a $k$-uniform hypergraph on $n$ vertices with vertex partition $A\cup B=V(H')$, such that $||A|-|B||\leq 2$ and $|A|$ and $n/k$ have different parity. Let $H'$ have edge set consisting of all $k$-element subsets of $V(H')$ that intersect $A$ in an odd number of vertices. Observe that $H'$ has no perfect matching, and that for every $1\leq d\leq k-1$ we have that $\delta_d(H')=(1/2 + o(1))\binom{n-d}{k-d}$. The second term in the lower bound is given by the hypergraph $H(n/k)$ defined as follows. Let $H(s)$ be the $k$-uniform hypergraph on $n$ vertices with edge set consisting of all $k$-element subsets of $V(H(s))$ intersecting a given (fixed) subset of $V(H(s))$ of size $s-1$, that is $H(s)=K_n^{(k)}-K_{n-s+1}^{(k)}$.

For $d=k-1$, $m_{k-1}^{n/k}(k,n)$ was determined exactly for large $n$ by R\"{o}dl, Ruci\'nski and Szemer\'edi~\cite{RRSone}. This was generalized by Treglown and Zhao~\cite{TZ}, who determined the extremal families for all $d\geq k/2$. The extremal constructions are similar to the parity based one of $H'$ above. 
This improves asymptotic bounds in \cite{Pik, RRStwo, RRSone}. Recently, Keevash, Knox and Mycroft~\cite{KKM} investigated the structure of hypergraphs whose minimum $(k-1)$-degree lies below the threshold and which have no perfect matching.

For $d<k/2$ less is known. In \cite{Large matchings} Conjecture \ref{main conjecture} was proved for $k-4\leq d\leq k-1$, by reducing it to a probabilistic conjecture of Samuels. In particular, this implies Conjecture \ref{main conjecture} for $k\leq 5$. Khan~\cite{Khanone}, and independently K\"uhn, Osthus and Treglown~\cite{KOT}, determined $m_1^{n/k}(k,n)$ exactly for $k=3$. Khan~\cite{Khantwo} also determined $m_1^{n/k}(k,n)$ exactly for $k=4$. It was shown by H\`an, Person and Schacht~\cite{Absorbing} that for $k\geq 3$, $1\leq d< k/2$ we have $m_d^{n/k}(k,n)\leq ((k-d)/k+o(1))\binom{n-d}{k-d}$. (The case $d=1$ of this is already due to Daykin and H\"aggkvist~\cite{DH}.) These bounds were improved by Markstr\"om and Ruci\'nski~\cite{MR}, using similar techniques, to
\begin{equation*}
m_d^{n/k}(k,n)\leq \left(\frac{k-d}{k}-\frac{1}{k^{k-d}}+o(1)\right)\binom{n-d}{k-d}.
\end{equation*}
Our main result improves on this bound, using quite different techniques.
\begin{theorem}\label{minimum degree theorem actual perfect}
Let $n$ and $1\leq d<k/2$ be such that $n$, $k$, $d$, $n/k\in \mathbb{N}$. Then
$$m_d^{n/k}(k,n)\leq \left(\frac{k-d}{k}-\frac{k-d-1}{k^{k-d}} +o(1)\right)\binom{n-d}{k-d}.$$
\end{theorem}

We also consider degree conditions that force smaller matchings. As a consequence of the results of K\"uhn, Osthus and Treglown as well as those of Khan mentioned above, $m_1^s(k,n)$ is determined exactly whenever $s\leq n/k$ and $k\leq 4$ (for details see the concluding remarks in \cite{KOT}). More generally, we propose the following version of Conjecture \ref{main conjecture} for non-perfect matchings.
\begin{conjecture}\label{partial main conjecture}
For all $\varepsilon>0$ and all integers $n$, $d$, $k$, $s$ with $1\leq d\leq k-1$ and $0\leq s\leq (1-\varepsilon)n/k$ we have
$$m_d^{s}(k,n)= \left(1-\left(1-\frac{s}{n}\right)^{k-d}+o(1)\right)\binom{n-d}{k-d}.$$ 
\end{conjecture}
In fact it may be that the bound holds for all $s\leq n-C$, for some $C$ depending only on $d$ and $k$. The lower bound here is given by $H(s)$. The case $d=k-1$ of Conjecture \ref{partial main conjecture} follows%
easily from the determination of $m_{k-1}^s(k,n)$ for $s$ close to $n/k$ in \cite{RRSone}. Bollob\'as, Daykin and Erd\H{o}s~\cite{BDE} determined $m_1^s(k,n)$ for small $s$, i.e. whenever $s<n/2k^3$. For $1\leq d\leq k-2$ we are able to determine $m_d^s(k,n)$ asymptotically for non-perfect matchings of any size at most $n/2(k-d)$. Note that this proves Conjecture \ref{partial main conjecture} in the case $k/2\leq d\leq k-2$, say.
\begin{theorem}\label{main theorem}
Let $\varepsilon >0$ and let $n$, $k$, $d$ be integers with $1\leq d\leq k-2$, and let $0\leq a< \min\{1/2(k-d),(1-\varepsilon)/k\}$ be such that $an\in \mathbb{N}$. Then
$$m^{an}_d(k,n)= \left(1-(1-a)^{k-d}+ o(1)\right)\binom{n-d}{k-d}.$$
\end{theorem}

\subsection{Large matchings in hypergraphs with many edges}

In proving Theorem \ref{main theorem} it will be useful for us to consider the following related problem. A classical theorem of Erd\H{o}s and Gallai~\cite{EG} determines the number of edges in a graph which forces a matching of a given size. In 1965, Erd\H{o}s~\cite{ErdConj} made a conjecture which would generalize this to $k$-uniform hypergraphs.
\begin{conjecture}\label{Erdos main conjecture}
Let $n$, $k\geq 2$ and $1\leq s\leq n/k$ be integers. Then
$$m_0^s(k,n)=\max\left\{\binom{ks-1}{k}, \binom{n}{k}-\binom{n-s+1}{k}\right\}+1.$$
\end{conjecture}

For $k=3$ this conjecture was verified by Frankl~\cite{Frankltwo}. For the case $k= 4$, Conjecture~\ref{Erdos main conjecture} was verified asymptotically by Alon, Frankl, Huang, R\"odl, Ruci\'nski and Sudakov~\cite{Large matchings}. Recently, Frankl confirmed the conjecture exactly for $s\leq n/2k$, i.e. when the aim is to cover at most half of the vertices of the hypergraph.
\begin{theorem}\label{Frankl theorem}\cite{Frankl}
Let $n$, $k$, $s\in \mathbb{N}$ be such that $n$, $k\geq 2$ and $n\geq (2s-1)k-s+1$. Then
$$m_0^s(k,n)=\binom{n}{k}-\binom{n-s+1}{k}+1.$$
\end{theorem}

\subsection{Large fractional matchings}

Our approach to proving our results uses the concepts of fractional matchings and fractional vertex covers. A \textit{fractional matching} in a $k$-uniform hypergraph $G=(V,E)$ is a function $w: E \rightarrow [0,1]$ of weights of edges, such that for each $v\in V$ we have $\sum_{e\in E: v\in e} w(e)\leq 1$. The \textit{size} of $w$ is $\sum_{e\in E} w(e)$. We say $w$ is \emph{perfect} if it has size $|V|/k$. A \textit{fractional vertex cover} in $G$ is a function $w: V \rightarrow [0,1]$ of weights of vertices, such that for each $e\in E$ we have $\sum_{v\in e} w(v)\geq 1$. The \textit{size} of $w$ is $\sum_{v\in V} w(v)$.

A key idea (already used e.g. in \cite{Large matchings, RRStwo}) is that we can switch between considering the largest fractional matching and the smallest fractional vertex cover of a hypergraph. The determination of these quantities are dual linear programming problems, and hence by the Duality Theorem they have the same size.

For $s\in \mathbb{R}$ we let $f_d^s(k,n)$ denote the minimum integer $m$ such that every $k$-uniform hypergraph $G$ on $n$ vertices with $\delta_d(G)\geq m$ has a fractional matching of size $s$. It was shown in \cite{RRStwo} that $f_{k-1}^{n/k}(k,n)=\lceil n/k\rceil$.

To prove Theorem~\ref{main theorem}, we use Theorem \ref{Frankl theorem}, along with methods similar to those developed in \cite{Large matchings}, to convert the edge-density conditions for the existence of matchings into corresponding minimum degree conditions for the existence of fractional matchings (see Proposition~\ref{main proposition}). We then use the Weak Hypergraph Regularity Lemma to prove Theorem~\ref{main theorem} by converting our fractional matchings into integer ones. Our argument also gives the following theorem which, for $1\leq d\leq k-2$, asymptotically determines $f_d^s(k,n)$ for fractional matchings of any size up to $n/2(k-d)$. Note that this determines $f_d^s(k,n)$ asymptotically for all $s\in (0,n/k)$ whenever $d\geq k/2$.
\begin{theorem}\label{minimum degree theorem}
Let $n$, $k\geq 3$, and $1\leq d\leq k-2$ be integers and let $0\leq a\leq \min\{1/2(k-d), 1/k\}$. Then
$$f^{an}_d(k,n)= \left(1-(1-a)^{k-d}+ o(1)\right)\binom{n-d}{k-d}.$$
\end{theorem}
We prove Theorem \ref{minimum degree theorem actual perfect} in a similar fashion, via the following two theorems.
\begin{theorem}\label{edge-density perfect theorem}
Let $n$, $k\geq 2$, $d\geq 1$ be integers. Then
$$f_0^{n/(k+d)}(k,n)\leq \left(\frac{k}{k+d}-\frac{k-1}{(k+d)^{k}} +o(1)\right)\binom{n}{k}.$$
\end{theorem}
\begin{theorem}\label{minimum degree theorem perfect}
Let $n$, $k\geq 3$, $1\leq d\leq k-2$ be integers. Then
$$f_d^{n/k}(k,n)\leq \left(\frac{k-d}{k}-\frac{k-d-1}{k^{k-d}} +o(1)\right)\binom{n-d}{k-d}.$$
\end{theorem}

The rest of the paper is organised as follows. In Section \ref{Notation, tools and preliminary results} we lay out some notation, set out some useful tools, and prove some preliminary results. Section~\ref{Minimum edge-density conditions for fractional matchings} is the heart of the paper, in which we prove Theorem \ref{edge-density perfect theorem}. In Section \ref{Minimum vertex degree conditions for fractional matchings} we derive Theorems \ref{minimum degree theorem} and \ref{minimum degree theorem perfect}, and in Section \ref{Constructing integer matchings from fractional ones} we derive Theorems \ref{minimum degree theorem actual perfect} and \ref{main theorem}. In Section~\ref{concluding remarks} we give some concluding remarks.

\section{Notation, tools and preliminary results}\label{Notation, tools and preliminary results}

\subsection{Notation}
Since in many of the proofs in this paper we often consider vertex degrees, when $S=\{v\}$ is a set containing only one vertex we write $d_G(v)$ to denote $deg_G(S)$ and we refer to $d_G(v)$ as the degree of $v$ (in $G$). We let $e(G)$ denote the number of edges in a hypergraph $G$, and let $|G|$ denote the number of its vertices. For a set $V$ and a positive integer $k$ we let $\binom{V}{k}$ denote the set of all $k$-element subsets of $V$. For $m\in \mathbb{N}$ we let $[m]$ denote the set $\{1,\dots, m\}$. Whenever we refer to a $k$-tuple, we assume that it is unordered. Given a hypergraph $G=(V,E)$ and a set $S\subseteq V$, we refer to the pair $(V\backslash S, \{e\subseteq V: S\cap e=\emptyset,\hspace{1mm} e\cup S\in E\})$ as the \textit{neighbourhood hypergraph} of $S$ (in $G$). If $S=\{v\}$ has just one element then we may refer to this pair as the neighbourhood hypergraph of $v$. For $U\subseteq V$ we denote by $G[U]$ the hypergraph induced by $U$ on $G$, that is the hypergraph with vertex set $U$ and edge set $\{e\in E: e\subseteq U\}$.

\subsection{Tools and preliminary results}
In proving some of our results we will use the lower bound given by the earlier construction $H(s)$, for all integers $n$, $d$, $k$, $s$ with $k\geq 2$ and $0\leq d\leq k-1$ and $0\leq s\leq n/k$:
\begin{equation}\label{fractional matching general lower bound}
m_d^s(k,n)\geq f_d^s(k,n)\geq \left(1-(1-s/n)^{k-d}+o(1)\right)\binom{n-d}{k-d}.
\end{equation}

Now, as mentioned in Section \ref{Introduction}, a key tool in this paper is that the determination of the size of the largest fractional matching of a $k$-uniform hypergraph is a linear programming problem, and its dual problem is to determine the size of the smallest fractional vertex cover of the hypergraph. The following proposition, which follows by the Duality Theorem, will be very useful to us.
\begin{proposition}\label{duality}
Let $k\geq 2$ and let $G$ be a $k$-uniform hypergraph. The size of the largest fractional matching of $G$ is equal to the size of the smallest fractional vertex cover of $G$.
\end{proposition}
In the rest of this section we collect some preliminary results.
\begin{proposition}\label{sums minimums vertex covers proposition}
Let $G=(V,E)$ be a hypergraph, $E'\subseteq E$, $S\subseteq V$, and let $w$ be a fractional vertex cover of $G$. Then
$$e(G)\leq \sum\limits_{e\in E} \sum\limits_{v\in e\backslash S} w(v)+ \sum\limits_{e\in E'} \sum\limits_{v\in e\cap S} w(v)+ |E\backslash E'|.$$
\end{proposition}
\begin{proof}
As $w$ is a fractional vertex cover of $G$,
\begin{align*}
e(G)&= |E'|+|E\backslash E'| \leq \sum\limits_{e\in E'} \sum\limits_{v\in e} w(v)+ |E\backslash E'|
&\leq \sum\limits_{e\in E} \sum\limits_{v\in e\backslash S} w(v)+ \sum\limits_{e\in E'} \sum\limits_{v\in e\cap S} w(v)+ |E\backslash E'|.
\end{align*}
\end{proof}
The following crude bound will sometimes be useful. The proof is immediate from the definitions.
\begin{proposition}\label{smooth proposition}
Suppose that $k\geq 2$ and $0<a,c<1$ are fixed. Then for every $\varepsilon >0$ there exists $n_0=n_0(k,\varepsilon)$ such that if $n\geq n_0$ and $f_0^{an}(k,n)\leq c\binom{n}{k}$ then\COMMENT{COMMENT: Proof of Proposition \ref{smooth proposition}:
Suppose $G$ is a $k$-uniform hypergraph on $n$ vertices with $e(G)\geq (c+\varepsilon)\binom{n}{k}$. Choose an arbitrary hyperedge $e\in E(G)$, and delete all edges incident to any $v\in e$, to form the new hypergraph $G'$. Then $$e(G')\geq e(G)-k\binom{n-1}{k-1}\geq (c+\varepsilon)\binom{n}{k}-k\binom{n-1}{k-1}\geq c\binom{n}{k},$$
where the last inequality holds as $n_0$ is sufficiently large. So by assumption, $G'$ has a fractional matching $M$ of size $an$. Note then that $M\cup \{e\}$ is a fractional matching of $G$. So indeed $G$ has a fractional matching of size $an+1$, as required.} $f_0^{an+1}(k,n)\leq (c+\varepsilon)\binom{n}{k}$.
\end{proposition}
In the next section we will prove Theorem \ref{edge-density perfect theorem} by induction. For this we will need Theorem~\ref{base step}, which will establish the base case of this induction. Theorem~\ref{base step} is an easy consequence of the Erd\H{o}s-Gallai Theorem from \cite{EG}.
\begin{theorem}\label{base step}
For $k= 2$ and $x\leq 1/3$ we have
$$f^{xn}_0(k,n)= \left(1-(1-x)^{k}+o(1)\right)\binom{n}{k}.$$
\end{theorem}

The next proposition will also be needed in the proof of Theorem \ref{edge-density perfect theorem}. To prove this proposition we will need a well-known theorem of Baranyai~\cite{Baranyai} from 1975.
\begin{theorem}[Baranyai's Theorem]\label{Baranyai Theorem}
If $n\in \ell \mathbb{N}$ then the complete $\ell$-uniform hypergraph on $n$ vertices decomposes into edge-disjoint perfect matchings.
\end{theorem}
\begin{proposition}\label{Baranyai Corollary}
Let $n$, $k$, $\ell$ be integers with $k\geq 2$ and $1\leq \ell \leq k$, and let $\eta \in [0,1)$. Let $V$ be a set of size $n$. Suppose $S\subseteq V$, with $|S|\in \ell \mathbb{N}$. Then there exists $\tilde{E}\subseteq \{e\in \binom{V}{k}:|e\cap S|=\ell\}$ such that for every $v\in S$,
\begin{equation}\label{baranyai equation}
|\{e\in \tilde{E}: v\in e\}|= \left\lfloor \eta \binom{|S|}{\ell -1} \binom{n-|S|}{k-\ell}\right\rfloor.
\end{equation}
\end{proposition}
\begin{proof}
The cases where $\ell=1$ or $\eta =0$ are trivial. So suppose that $2\leq \ell \leq k$ and $\eta \in (0,1)$. Apply Theorem \ref{Baranyai Theorem} to find a decomposition of the complete $\ell$-uniform hypergraph on $S$ into edge-disjoint perfect matchings $M_1,\dots, M_{\binom{|S|-1}{\ell -1}}$.

We now construct $\tilde{E}$ by adding $k$-tuples from $\{e\in \binom{V}{k}:|e\cap S|=\ell\}$ greedily, under the following constraints:
\begin{enumerate}[(i)]
\item for all $i\in \{1,\dots,\binom{|S|-1}{\ell -1}\}$, we do not add any $k$-tuples in $\{e\in \binom{V}{k}: e\cap S\in M_{i+1}\}$ unless we have already added all $k$-tuples in $\{e\in \binom{V}{k}: e\cap S\in M_{i}\}$;
\item for every $v\in S$,
$$|\{e\in \tilde{E}: v\in e\}|\leq \eta \binom{|S|}{\ell -1} \binom{n-|S|}{k-\ell}.$$
\end{enumerate}
It is clear that {\rm (i)} and {\rm (ii)} ensure that the set $\tilde{E}$ obtained in this way satisfies (\ref{baranyai equation}) for every $v\in S$.
\end{proof}
\section{Minimum edge-density conditions for fractional matchings}\label{Minimum edge-density conditions for fractional matchings}
We will use the following lemma to prove Theorem \ref{edge-density perfect theorem} inductively.
\begin{lemma}\label{high degree vertices lemma}
Let $k\geq 3$ be fixed. Suppose that $a \in (0,1/(k+1)]$, $c\in (0,1)$ and that there exists $n_0\in \mathbb{N}$ such that for all $n\geq n_0$ we have
\begin{equation}\label{assumption in statement of high degree vertcies lemma}
f^{an/(1-a)}_0(k-1,n)\leq c\binom{n}{k-1}.
\end{equation}
Then for all $\varepsilon >0$ there exists $n_1\in \mathbb{N}$ such that for all $n\geq n_1$ any $k$-uniform hypergraph $G$ on $n$ vertices with at least $an$ vertices of degree at least
$$D:=\left(c(1-a)^{k-1}+\left(1-(1-a)^{k-1}\right)+ \varepsilon \right)\binom{n-1}{k-1}$$
has a fractional matching of size $an$.
\end{lemma}
\begin{proof}
Let $\varepsilon >0$ and choose $n_1$ sufficiently large. Consider a $k$-uniform hypergraph $G= (V,E)$ on $n$ vertices with at least $an$ vertices of degree at least $D$. Let $Y\subseteq V$ be the set of $\lceil an \rceil$ vertices of highest degree. Let $w$ be a fractional vertex cover of $G$ of least size. Consider the vertex $v_0\in Y$ with the lowest weight $w(v_0)$. Let $H$ be the neighbourhood hypergraph of $v_0$ in $G$. So
$$e(H)= d_G(v_0)\geq D= \left(c(1-a)^{k-1}+\left(1-(1-a)^{k-1}\right)+ \varepsilon \right)\binom{n-1}{k-1}.$$
Let $H':= H[V\backslash Y]$. Since the number of edges in $H$ with at least one vertex in $Y$ is at most $(1-(1-a)^{k-1}+o(1))\binom{n-1}{k-1}$, it follows that
\begin{align*}
e(H')&\geq e(H)- \left(1-(1-a)^{k-1}+o(1)\right)\binom{n-1}{k-1} \geq \left(c(1-a)^{k-1}+\varepsilon/2\right)\binom{n-1}{k-1}\\
&\geq (c+\varepsilon/3)\binom{|H'|}{k-1},
\end{align*}
where in the last two inequalities we use that $n_1$ was chosen sufficiently large. Note that $|H'|\geq n/2$, so we may assume that $|H'|\geq n_0$. Now, (\ref{assumption in statement of high degree vertcies lemma}) and Proposition \ref{smooth proposition} together imply that $H'$ has a fractional matching of size
$$a|H'|/ (1-a)+1= a(n-\lceil an \rceil)/ (1-a)+1\geq an.$$
So let $M$ be a fractional matching of $H'$ of size $an$. Note that for all $v\in V\backslash Y$, 
$$\sum\limits_{e\in E(H'): v\in e} M(e)\leq 1.$$
So we have that
$$\sum\limits_{v\in V}w(v) \geq \sum\limits_{v\in Y}w(v) + \sum\limits_{e\in E(H')}\sum\limits_{v\in e} M(e)w(v).$$
By the minimality of $w(v_0)$, this implies that
\begin{align*}
\sum\limits_{v\in V}w(v)&\geq anw(v_0)+ \sum\limits_{e\in E(H')}\sum\limits_{v\in e} M(e)w(v)= \sum\limits_{e\in E(H')} M(e)w(v_0)+ \sum\limits_{e\in E(H')}\sum\limits_{v\in e} M(e)w(v)\\
&= \sum\limits_{e\in E(H')} M(e)\left(w(v_0)+ \sum\limits_{v\in e}w(v)\right)\geq \sum\limits_{e\in E(H')} M(e)= an.
\end{align*}
The last inequality holds because by definition of $H'$ we have $e\cup \{v_0\}\in E$ for all $e\in E(H')$, and so $w(v_0)+\sum_{v\in e} w(v)\geq 1$. 

Hence the size of $w$ is at least $an$, so by Proposition \ref{duality} the largest fractional matching in $G$ has size at least $an$.
\end{proof}

The proof of Theorem~\ref{edge-density perfect theorem} proceeds as follows. Suppose $G$ has no fractional matching of size $n/(k+d)$. Then we use Lemma~\ref{high degree vertices lemma} and induction to show that $G$ contains few vertices of high degree. Moreover, by duality we show that $G$ has a small fractional vertex cover. We combine these two facts to show that the number of edges of $G$ does not exceed the expression stated in Theorem~\ref{edge-density perfect theorem}.

\removelastskip\penalty55\medskip\noindent{\bf Proof of Theorem~\ref{edge-density perfect theorem}.}
The proof will proceed by induction on $k$. The base step, $k=2$, follows\COMMENT{COMMENT:
$$1-\left(1-\frac{1}{2+d}\right)^2=\frac{2}{2+d}-\frac{2-1}{(2+d)^2}$$
rearranges to
$$\frac{(2+d)^2}{(2+d)^2}-\frac{(2+d-1)^2}{(2+d)^2}=\frac{2(2+d)}{(2+d)^2}-\frac{2-1}{(2+d)^2}$$
which rearranges to
$$d(2+d)-(1+d)^2+1=0,$$
which clearly holds for all positive $d$.} by Theorem~\ref{base step}, setting $x:=1/(2+d)$.

Now consider some $k> 2$ and suppose that the theorem holds for all smaller values of $k$. Fix $d\geq 1$. Let $\varepsilon >0$ and let $n_0\in \mathbb{N}$ be sufficiently large compared to $1/\varepsilon$, $k$ and $d$. For convenience let us define
$$\xi := \left(\frac{k-1}{k+d-1}-\frac{k-2}{(k+d-1)^{k-1}} \right)\left(\frac{k+d-1}{k+d}\right)^{k-1}+ \left(1-\left(\frac{k+d-1}{k+d}\right)^{k-1}\right)<1.$$
Consider any $k$-uniform hypergraph $G=(V,E)$ on $n\geq n_0$ vertices, and suppose that the largest fractional matching of $G$ is of size less than $n/(k+d)$. Then by Proposition \ref{duality} there exists a fractional vertex cover, $w$ say, of $G$ with size less than $n/(k+d)$. Let $a:=1/(k+d)$. So $a/(1-a)=1/(k+d-1)$. Let
$$c:=\frac{k-1}{k+d-1}-\frac{k-2}{(k+d-1)^{k-1}}+\varepsilon /4.$$
Then by induction,
$$f_0^{n'/(k+d-1)}(k-1,n)\leq c\binom{n'}{k-1},$$
for all sufficiently large $n'$. Thus, as $n_0$ is sufficiently large, Lemma \ref{high degree vertices lemma} implies that there are less than $n/(k+d)$ vertices of $G$ with degree at least $(\xi +\varepsilon/2) \binom{n-1}{k-1}$.

Let $S$ be the set of $|S|$ vertices of $G$ with highest degree, where $|S|\in k!\mathbb{N}$ is minimal such that $|S|\geq n/(k+d)$. So $d_G(v)<( \xi +\varepsilon/2 )\binom{n-1}{k-1}$ for all $v\in V\backslash S$. For every $i\in \{0,\dots,k\}$ let $S_i:=\{e\in \binom{V}{k}: |e\cap S|=i \}$. Given $X\subseteq \binom{V}{k}$, for all $v\in V$ let $t_X(v):= |\{e\in X: v\in e\}|$. Note that for all $v\in S$ the value of $t_{S_i}(v)$ is the same and $t_{S_0}(v)=0$. Let $\ell \in \{0,\dots,k\}$ be maximal such that for any $v\in S$ we have $\sum_{i=0}^{\ell-1}t_{S_i}(v)\leq \xi \binom{n-1}{k-1}$. Let $E''':= \binom{V}{k}\backslash S_k$. Then for each $v\in S$,
\begin{equation}\label{E'''}
t_{E'''}(v)= \left(1-\frac{1}{(k+d)^{k-1}} +o(1)\right)\binom{n-1}{k-1}> \xi \binom{n-1}{k-1}.
\end{equation}
The final inequality holds here for sufficiently large $n_0$, as it rearranges\COMMENT{COMMENT:
$$\left(1-\frac{1}{(k+d)^{k-1}}\right)> \left(\frac{k-1}{k+d-1}-\frac{k-2}{(k+d-1)^{k-1}} \right)\left(\frac{k+d-1}{k+d}\right)^{k-1}+ \left(1-\left(\frac{k+d-1}{k+d}\right)^{k-1}\right)$$
rearranges to
$$(k+d)^{k-1}-1> (k-1)(k+d-1)^{k-2}-(k-2)+(k+d)^{k-1}-(k+d-1)^{k-1}$$
which rearranges to
$$1< d(k+d-1)^{k-2}+(k-2).$$} to $d(k+d-1)^{k-2}+ (k-2) +o(1)> 1$. This shows that $\ell \leq k-1$. Let
$$\eta:= \left(\xi \binom{n-1}{k-1}-\sum\limits_{i=1}^{\ell-1}t_{S_i}(v)\right)/\binom{|S|}{\ell-1}\binom{n-|S|}{k-\ell}.$$
So $\eta \in [0,1)$. Apply Proposition \ref{Baranyai Corollary} with parameters $n$, $k$, $\ell$, $\eta$ to obtain a set $\tilde{E}\subseteq S_\ell$ such that for every $v\in S$,
$$t_{\tilde{E}}(v)=\left\lfloor \eta \binom{|S|}{\ell -1} \binom{n-|S|}{k-\ell}\right\rfloor.$$
Let $E'':=\bigcup_{i=0}^{\ell-1}S_i\cup \tilde{E}$. Then each $v\in S$ satisfies
\begin{equation}\label{t_E''}
t_{E''}(v)= \left\lfloor \xi \binom{n-1}{k-1}\right\rfloor.
\end{equation}
We can now give a lower bound on the size of $E''$ as follows: for each vertex $v\in S$ we count the number of $k$-tuples in $E''$ that contain $v$, and then adjust for the $k$-tuples that contain several vertices of $S$ and were thus counted several times as a result. Since $S_0\subseteq E''$ this yields
$$|E''|\stackrel{(\ref{t_E''})}{\geq} \left\lfloor \xi \binom{n-1}{k-1} \right\rfloor \frac{n}{k+d}+|S_0|-\sum\limits_{j=1}^{k-1}(j-1)|S_j|.$$
Note that since $E''\subseteq E'''$ we only need to consider values of $j$ up to $k-1$ in the summation, rather than $k$. Now, note that
\begin{align*}
|S_0|-\sum\limits_{j=1}^{k}(j-1)|S_j|&=\binom{n}{k}-\sum\limits_{v\in S}\binom{n-1}{k-1}=\binom{n}{k}-\left(\frac{n}{k+d}+o(1)\right)\binom{n-1}{k-1}\\
&=\left(1-\frac{k}{k+d}+o(1)\right)\binom{n}{k}.
\end{align*}
Hence, as $(k-1)|S_k|=((k-1)/(k+d)^{k}+o(1))\binom{n}{k}$,
\begin{equation}\label{E'' size}
|E''|\geq (\xi + o(1))\binom{n-1}{k-1} \frac{n}{k+d}+\left(1-\frac{k}{k+d}+\frac{k-1}{(k+d)^{k}}+o(1)\right)\binom{n}{k}.
\end{equation}
Now, let $E':= E\cap E''$. Also, note that by Proposition \ref{sums minimums vertex covers proposition},
\begin{align*}
e(G)\leq \sum\limits_{e\in E} \sum\limits_{v\in e\backslash S} w(v)+ \sum\limits_{e\in E'} \sum\limits_{v\in e\cap S} w(v)+ |E\backslash E'|.
\end{align*}
Recall that $d_G(v)<(\xi +\varepsilon/2 )\binom{n-1}{k-1}$ for all $v\in V\backslash S$ and that by (\ref{t_E''}) the number of edges in $E'$ incident to $v$ is at most $\xi \binom{n-1}{k-1}$ for all $v\in S$. So
$$e(G)\leq \sum\limits_{v\in V} (\xi +\varepsilon/2)\binom{n-1}{k-1} w(v)+ |E\backslash E'|.$$
Now note that $|E\backslash E'|\leq |\binom{V}{k}\backslash E''|=\binom{n}{k}-|E''|$ and recall that the size of $w$ is less than $n/(k+d)$. So
\begin{eqnarray*}
e(G)&< &(\xi +\varepsilon/2)\binom{n-1}{k-1} \frac{n}{k+d}+ \binom{n}{k}-|E''|\\
&\stackrel{(\ref{E'' size})}{\leq} &(\xi +\varepsilon/2)\binom{n-1}{k-1} \frac{n}{k+d}-(\xi + o(1))\binom{n-1}{k-1} \frac{n}{k+d}\\
&&+\left(\frac{k}{k+d}-\frac{k-1}{(k+d)^{k}}+o(1)\right)\binom{n}{k} \\
&\leq & \left(\frac{k}{k+d}-\frac{k-1}{(k+d)^{k}} +\varepsilon \right)\binom{n}{k}.
\end{eqnarray*}
The final inequality holds since $n_0$ is sufficiently large. By definition, this shows that
$$f_0^{n/(k+d)}(k,n)\leq \left(\frac{k}{k+d}-\frac{k-1}{(k+d)^{k}} +o(1)\right)\binom{n}{k}.$$ This completes the inductive step and hence the proof.
\endproof
\section{Minimum degree conditions for fractional matchings}\label{Minimum vertex degree conditions for fractional matchings}
The following proposition generalises Proposition 1.1 in \cite{Large matchings}, with a similar proof idea\COMMENT{COMMENT: Note that in the proposition we always have that $an(k-d)\leq n-d$, since $a\leq 1/k$. So the expression $f_0^{an}(k-d,n-d)$ makes sense}. It allows us to transform bounds involving edge densities into bounds involving $d$-degrees.
\begin{proposition}\label{main proposition}
Let $\varepsilon \geq 0$, let $k$, $d$, $n$ be integers with $n\geq k\geq 3$, $1\leq d\leq k-2$, and $d<(1-\varepsilon^{1/d})n$. Let $a\in [0,(1-\varepsilon^{1/d})/k]$. Suppose $H$ is a $k$-uniform hypergraph on $n$ vertices, such that for at least $(1-\varepsilon)\binom{n}{d}$ $d$-tuples of vertices $L\in \binom{V(H)}{d}$ we have
$$deg_H(L)\geq f_0^{an}(k-d,n-d).$$
Then $H$ has a fractional matching of size $an$.
\end{proposition}
\begin{proof}
The outline of the proof goes as follows. We will assume that there is no fractional matching of size $an$ in a $k$-uniform hypergraph $H=(V,E)$ on $n$ vertices and then show that for more than $\varepsilon\binom{n}{d}$ $d$-tuples of vertices $L\in \binom{V}{d}$, the neighbourhood hypergraph $H(L)$ of $L$ in $H$ has no fractional matching of size $an$. This will imply that for more than $\varepsilon\binom{n}{d}$ $d$-tuples of vertices $L$, $deg_H(L)=e(H(L))< f_0^{an}(k-d,n-d)$. This will prove the result in contrapositive.

So suppose $H=(V,E)$ is an $n$-vertex $k$-uniform hypergraph, with no fractional matching of size $an$. Then by Proposition \ref{duality}, $H$ has a fractional vertex cover, $w$ say, of size less than $an$. Let 
$$E_w:= \left\{e\in \binom{V}{k}: \sum\limits_{v\in e} w(v) \geq 1\right\},$$
and let $H_w:=(V,E_w)$. Since $H\subseteq H_w$ we can, without loss of generality, replace $H$ with $H_w$. Let $U\subseteq V$ be the set of $\lfloor \varepsilon^{1/d}n\rfloor+d$ vertices of smallest weights. Let $\mathcal{L}:= \binom{U}{d}$. Note that
$$|\mathcal{L}|=\binom{\lfloor \varepsilon^{1/d} n\rfloor+d}{d}> \frac{(\varepsilon^{1/d} n)^d}{d!}=\varepsilon \frac{n^d}{d!} \geq \varepsilon \binom{n}{d}.$$
Consider any $L\in \mathcal{L}$. Let $H_w(L)$ be the neighbourhood hypergraph of $L$ in $H_w$. We will show that $H_w(L)$ has no fractional matching of size $an$. Without loss of generality we may assume that the elements of $L$ all have equal weights, $w(L)$ say. (If not, we could replace these weights by their average, which would alter neither $\sum_{v\in V}w(v)$ nor $\sum_{v\in e}w(v)$ for any $e\supseteq L$. These are the only two quantities involving weights that we will consider in what follows.) Observe that $w(L)<1/k$, else the size of $w$ would be at least
$$\frac{n(1-\varepsilon^{1/d})}{k} \geq an.$$
We now define a new weight function $w'(v)$ on the vertices in $V$:
$$w'(v):=\min \left\{ \max \left\{ 0, w^*(v) \right\}, 1 \right\},\,\,\,\,\mathrm{ where }\,\,\,\,\,\,  w^*(v):=\frac{w(v)-w(L)}{1-kw(L)}.$$
Note that only for vertices $u\in U\backslash L$ can it be that $w^*(u)<0$. Note also that since $w(v)\geq 0$ for all $v\in V$, we have that $w^*(u)\geq -w(L)/(1-kw(L))$ for such vertices $u$. Hence,
\begin{align*}
\sum\limits_{v\in V}w'(v)&\leq \left( \sum\limits_{v\in V} w^*(v)\right) + |U\backslash L|\frac{w(L)}{1-kw(L)}<\frac{an-nw(L)+\varepsilon^{1/d}nw(L)}{1-kw(L)}\\
&= an\frac{1-(1/a)(1-\varepsilon^{1/d})w(L)}{1-kw(L)}\leq an,
\end{align*}
and for any given $e\in \{e'\in E_w: e'\supseteq L\}$ we have that
$$\sum\limits_{v\in e}w'(v)\geq \min \left\{ \frac{\sum_{v\in e}w(v)-kw(L)}{1-kw(L)}, 1\right\} \geq \min \left\{ \frac{1-kw(L)}{1-kw(L)}, 1 \right\}= 1.$$
Moreover, $\sum_{v\in L}w'(v)=0$. It follows that the function $w'$ restricted to $V\backslash L$ is a fractional vertex cover of $H_w(L)$ of size less than $an$, and so by Proposition \ref{duality}, $H_w(L)$ has no fractional matching of size $an$, which completes the proof.
\end{proof}
We can now derive Theorems \ref{minimum degree theorem} and \ref{minimum degree theorem perfect}.
\removelastskip\penalty55\medskip\noindent{\bf Proof of Theorem~\ref{minimum degree theorem}.}
Let $k':= k-d$ and $n':= n-d$. Note that Theorem~\ref{Frankl theorem} implies that
\begin{equation}\label{Frankl equation}
m_0^{an}(k,n)=\left(1-(1-a)^k +o(1)\right)\binom{n}{k}
\end{equation}
for all $a\leq 1/2k$. Now Proposition~\ref{smooth proposition} implies that for all $0\leq a\leq \min\{1/2(k-d), 1/k\}$,
\begin{align*}
f_0^{an}(k-d,n-d)&\hspace{0.14cm}=f_0^{a(n'+d)}(k',n')\leq f_0^{an'+1}(k',n')\leq m_0^{an'}(k',n')+o(1)\binom{n'}{k'}\\
&\stackrel{(\ref{Frankl equation})}{=} \left( 1-(1-a)^{k'}+o(1)\right)\binom{n'}{k'}= \left( 1-(1-a)^{k-d}+o(1)\right)\binom{n-d}{k-d}.
\end{align*}
The upper bound in Theorem~\ref{minimum degree theorem} follows now from Proposition~\ref{main proposition} applied with $\varepsilon =0$. The lower bound follows from (\ref{fractional matching general lower bound}).
\endproof
\removelastskip\penalty55\medskip\noindent{\bf Proof of Theorem~\ref{minimum degree theorem perfect}.}
Let $k':= k-d$ and $n':= n-d$. Then Theorem \ref{edge-density perfect theorem} and Proposition~\ref{smooth proposition} together imply that
\begin{align*}
f_0^{n/k}(k-d,n-d)&=f_0^{(n'+d)/(k'+d)}(k',n')\leq f_0^{n'/(k'+d)+1}(k',n')\\
&\leq\left( \frac{k'}{k'+d}-\frac{k'-1}{(k'+d)^{k'}}+o(1)\right)\binom{n'}{k'}\\
&= \left( \frac{k-d}{k}-\frac{k-d-1}{k^{k-d}}+o(1)\right)\binom{n-d}{k-d}.
\end{align*}
So Theorem \ref{minimum degree theorem perfect} follows now from Proposition~\ref{main proposition} applied with $\varepsilon =0$.
\endproof
The case $\varepsilon >0$ of Proposition~\ref{main proposition} will be used in the next section.
\section{Constructing integer matchings from fractional ones}\label{Constructing integer matchings from fractional ones}
We will construct integer matchings from fractional ones using the Weak Hypergraph Regularity Lemma. Before stating this we will need the following definitions.

Given a $k$-tuple $(V_1,\dots, V_k)$ of disjoint subsets of the vertices of a $k$-uniform hypergraph $G=(V,E)$, we define $(V_1,\dots, V_k)_G$ to be the $k$-partite subhypergraph with vertex classes $V_1,\dots, V_k$ induced on $G$. We let
$$d_G(V_1,\dots, V_k)=\frac{e((V_1,\dots, V_k)_G)}{\prod_{i\in \{1,\dots, k\}}|V_i|}$$
denote the \textit{density} of $(V_1,\dots, V_k)_G$.
\begin{definition}[$\varepsilon$-regularity]
Let $\varepsilon>0$, let $G=(V,E)$ be a $k$-uniform hypergraph, and let $V_1,\dots, V_k\subseteq V$ be disjoint. We say that $(V_1,\dots, V_k)_G$ is $\varepsilon$-regular if for every subhypergraph $(V_1',\dots, V_k')_G$ with $V_i'\subseteq V_i$ and $|V_i'|\geq \varepsilon |V_i|$ for each $i\in \{1,\dots, k\}$, we have that
$$|d_G(V_1',\dots, V_k')-d_G(V_1,\dots, V_k)|< \varepsilon.$$
\end{definition}
The following result was proved by Chung~\cite{FRK}. The proof follows the lines of that of the original Regularity Lemma for graphs~\cite{SRL}.
\begin{lemma}[Weak Hypergraph Regularity Lemma]
For all integers $k\geq 2$, $L_0\geq 1$, and every $\varepsilon >0$ there exists $N=N(\varepsilon, L_0, k)$ such that if $G=(V,E)$ is a $k$-uniform hypergraph on $n \geq N$ vertices, then $V$ has a partition $V_0,\dots,V_L$ such that the following properties hold:
\begin{enumerate}[{\rm(i)}]
\item $L_0 \leq L \leq N$ and $|V_0| \leq \varepsilon n$,
\item $|V_1|=\dots=|V_L|$,
\item for all but at most $\varepsilon \binom{L}{k}$ $k$-tuples $\{i_1,\dots,i_k\}\in \binom{[L]}{k}$, we have that $(V_{i_1},\dots,V_{i_k})_{G}$ is $\varepsilon$-regular.
\end{enumerate}
\end{lemma}
We call the partition classes $V_1,\dots,V_L$ \textit{clusters}, and $V_0$ the \textit{exceptional set}. For our purposes we will in fact use the degree form of the Weak Hypergraph Regularity Lemma.
\begin{lemma}[Degree Form of the Weak Hypergraph Regularity Lemma]\label{DFWHRL}
For all integers $k \geq 2$, $L_0\geq 1$ and every $\varepsilon >0$, there is an $N=N(\varepsilon,L_0,k)$ such that for every $d\in[0,1)$ and for every hypergraph $G=(V,E)$ on $n\geq N$ vertices there exists a partition of $V$ into $V_0,V_1,\dots,V_L$ and a spanning subhypergraph $G'$ of $G$ such that the following properties hold:
\begin{enumerate}[{\rm (i)}]
\item $L_0\leq L \leq N$ and $|V_0|\leq \varepsilon n$,
\item $|V_1|=\dots=|V_L|=: m$,
\item $d_{G'}(v)>d_G(v)-(d+\varepsilon)n^{k-1}$ for all $v\in V$,
\item every edge of $G'$ with more than one vertex in a single cluster $V_i$, for some $i\in \{1,\dots, L\}$, has at least one vertex in $V_0$,
\item for all $k$-tuples $\{i_1,\dots,i_k\}\in \binom{[L]}{k}$, we have that $(V_{i_1},\dots,V_{i_k})_{G'}$ is $\varepsilon$-regular and has density either $0$ or greater than $d$.
\end{enumerate}
\end{lemma}
The proof is very similar to that of the degree form of the Regularity Lemma for graphs, so we omit it here; for details see \cite{Townsend PhD}.

We now define a type of hypergraph that will be essential in our application of the Weak Hypergraph Regularity Lemma.
\begin{definition}[Reduced Hypergraph]
Let $G=(V,E)$ be a $k$-uniform hypergraph. Given parameters $\varepsilon> 0$, $d \in [0,1)$ and $L_0\geq 1$ we define the reduced hypergraph $R=R(\varepsilon, d, L_0)$ of $G$ as follows. Apply the degree form of the Weak Hypergraph Regularity Lemma to $G$, with parameters $\varepsilon$, $d$, $L_0$ to obtain a spanning subhypergraph $G'$ and a partition $V_0,\dots,V_L$ of $V$, with exceptional set $V_0$ and clusters $V_1,\dots,V_L$. Then $R$ has vertices $V_1,\dots,V_L$, and there exists an edge between $V_{i_1},\dots,V_{i_k}$ precisely when $(V_{i_1},\dots,V_{i_k})_{G'}$ is $\varepsilon$-regular with density greater than $d$.
\end{definition}
The following lemma tells us that this reduced hypergraph (almost) inherits the minimum degree properties of the original hypergraph. The proof is similar to that of the well known version for graphs, but we include it here for completeness.
\begin{lemma}\label{reduced graph edges}
Suppose $c>0$, $k\geq 2$, $1\leq \ell \leq k-1$, $L_0\geq 1$, and $0<\varepsilon \leq d\leq c^3/64$. Let $G$ be a $k$-uniform hypergraph with $\delta_\ell (G)\geq c|G|^{k-\ell}$. Let $R=R(\varepsilon, d, L_0)$ be the reduced hypergraph of $G$. Then at least $\binom{|R|}{\ell}- d^{1/3}(2k)^\ell |R|^\ell$ of the $\ell$-tuples of vertices of $R$ have degree at least $(c-4d^{1/3})|R|^{k-\ell}$.
\end{lemma}
\begin{proof}
Let $G'$ be the spanning subhypergraph of $G$ obtained by applying the degree form of the Weak Hypergraph Regularity Lemma to $G$ with parameters $\varepsilon$, $d$, $L_0$; let $V_1,\dots,V_L$ denote the vertices of $R$, and let $m$ denote the size of these clusters.

First recall that given any vertex $x\in V(G')$ we know that $d_{G'}(x)>d_G(x)-(d+\varepsilon)|G|^{k-1}$. Note that since $|V_0|\leq \varepsilon |G|$, we have that the number of edges incident to $x$ that contain a vertex in $V_0$ is at most $\varepsilon |G| \binom {|G|}{k-2}\leq \varepsilon |G|^{k-1}$. Hence for all $v\in V(G'-V_0)$, we have that
$$d_{G'-V_0}(v)>d_G(v)-(d+2\varepsilon)|G|^{k-1}\geq d_G(v)-3d|G|^{k-1}.$$
We call an $\ell$-tuple $A$ of vertices of $G'-V_0$ \emph{bad} if $deg_{G'-V_0}(A)\leq deg_G(A)-3d^{1/3}|G|^{k-\ell}$. So for each $v\in V(G'-V_0)$ there are at most $\binom{k-1}{\ell-1} d^{2/3} |G|^{\ell -1}$ bad $\ell$-tuples $A$ with $v\in A$. (This follows by double-counting the number of pairs $(A,e)$ where $A$ is a bad $\ell$-tuple with $v\in A$ and $e\in E(G)\backslash E(G'-V_0)$ is an edge containing $A$.) This in turn implies that in total at most $\binom{k-1}{\ell-1} d^{2/3} |G|^{\ell}$ of the $\ell$-tuples $A$ are bad. Given $1\leq s\leq k$ and an $s$-tuple $(V_{i_1},\dots,V_{i_s})$ of clusters of $R$, we say that an $s$-tuple $A$ of vertices of $G'- V_0$ \emph{lies in} $(V_{i_1},\dots,V_{i_s})$ if $|A\cap V_{i_\alpha}|=1$ for all $\alpha \in \{1,\dots,s\}$. We call an $\ell$-tuple $(V_{i_1},\dots,V_{i_\ell})$ of clusters of $R$ \emph{nice} if there are less than $d^{1/3}m^\ell$ bad $\ell$-tuples $A$ of vertices of $G'-V_0$ which lie in $(V_{i_1},\dots,V_{i_\ell})$. So less than $\binom{k-1}{\ell-1} d^{1/3} |G|^{\ell}/m^{\ell}\leq d^{1/3}(2k)^\ell |R|^\ell$ of the $\ell$-tuples of clusters of $R$ are not nice. Hence it suffices to show that any nice $\ell$-tuple of clusters of $R$ has degree at least $(c-4d^{1/3})|R|^{k-\ell}$ in $R$.

Consider any nice $\ell$-tuple of clusters of $R$, say $(V_{i_1},\dots, V_{i_\ell})$. Let $\mathcal{A}$ denote the set of all $\ell$-tuples $A$ of vertices of $G'- V_0$ which lie in $(V_{i_1},\dots, V_{i_\ell})$ and are not bad. So $|\mathcal{A}|\geq (1-d^{1/3})m^\ell$. Moreover, the number of edges $e$ of $G'-V_0$ with $|e\cap V_{i_\alpha}|=1$ for all $\alpha \in \{1,\dots, \ell\}$ is at least
\begin{equation}\label{contradiction equation}
\sum\limits_{A\in \mathcal{A}} deg_{G'-V_0}(A)\geq |\mathcal{A}| \left(c-3d^{1/3}\right)|G|^{k-\ell}\geq \left(c-4d^{1/3}\right)|G|^{k-\ell}m^\ell.
\end{equation}
Now suppose that the degree in $R$ of $(V_{i_1},\dots, V_{i_\ell})$ is less than $(c-4d^{1/3})|R|^{k-\ell}$. Then the number of $(k-\ell)$-tuples $\{j_1,\dots, j_{k-\ell}\}\in \binom{[L]}{k-\ell}$ for which $(V_{i_1},\dots,V_{i_\ell}, V_{j_1},\dots,V_{j_{k-\ell}})_{G'}$ is $\varepsilon$-regular with density greater than $d$ is less than $(c-4d^{1/3})|R|^{k-\ell}$. Note that at most $m^k$ edges of $G'-V_0$ lie in such a subhypergraph. So the number of edges $e$ of $G'-V_0$ with $|e\cap V_{i_\alpha}|=1$ for all $\alpha \in \{1,\dots, \ell\}$ is less than
$$(c-4d^{1/3})|R|^{k-\ell}m^k\leq (c-4d^{1/3})|G|^{k-\ell}m^\ell,$$
contradicting (\ref{contradiction equation}). This completes the proof.
\end{proof}
The following lemma uses all of the previous results of this section to allow us to convert our fractional matchings into integer ones. We will use the notation $a\ll b$ to mean that we can find an increasing function $f$ for which all of the conditions in the proof are satisfied whenever $a\leq f(b)$.
\begin{lemma}\label{fractional to integer matchings}
Let $k\geq 2$ and $1\leq \ell \leq k-1$ be integers, and let $\varepsilon> 0$. Suppose that for some $b,c\in (0,1)$ and some integer $n_0$, any $k$-uniform hypergraph on $n\geq n_0$ vertices with at least $(1-\varepsilon)\binom{n}{\ell}$ $\ell$-tuples of vertices of degree at least $cn^{k-\ell}$ has a fractional matching of size $(b+\varepsilon)n$. Then there exists an integer $n_0'$ such that any $k$-uniform hypergraph $G$ on $n\geq n_0'$ vertices with $\delta_\ell (G)\geq (c+\varepsilon)n^{k-\ell}$ has an (integer) matching of size at least $bn$.
\end{lemma}
\begin{proof}
Define $n_0'\in \mathbb{N}$ and new constants $\varepsilon'$ and $d$ such that $0< 1/n_0'\ll \varepsilon' \ll d\ll \varepsilon, c, 1/k, 1/n_0$. Let $G$ be a $k$-uniform hypergraph on $n\geq n_0'$ vertices, with $\delta_\ell (G)\geq (c+\varepsilon)n^{k-1}$. Let $G'$ be the spanning subhypergraph of $G$ obtained by applying the degree form of the Weak Hypergraph Regularity Lemma to $G$ with parameters $\varepsilon'$, $d$, $n_0$. Let $R:=R(\varepsilon', d, n_0)$ be the corresponding reduced hypergraph, and let $L:= |R|$. By Lemma \ref{reduced graph edges} at least $(1-\varepsilon)\binom{L}{\ell}$ $\ell$-tuples of vertices of $R$ have degree at least
$$(c+\varepsilon-4d^{1/3})L^{k-\ell}\geq cL^{k-\ell}.$$
So by the assumption in the statement of the lemma, $R$ has a fractional matching, $F$ say, of size $(b+\varepsilon)L$.

For each $e\in E(R)$, let $K_e:=\lceil(1-2\varepsilon')F(e)m\rceil$, where $m$ is the size of each of the clusters of $R$. Now construct an integer matching, $M$ say, in $G$ by greedily adding to $M$ edges of $G'$ until, for each $e= \{V_{j_1},\dots,V_{j_k}\}\in E(R)$, $M$ contains precisely $K_e$ edges of $(V_{j_1},\dots,V_{j_k})_{G'}$. Note that at each stage of this process the number of vertices in each $V_i\in V(R)$ that would be covered by $M$ is at most
$$\sum\limits_{e: V_i\in e}K_e\leq \sum\limits_{e: V_i\in e}((1-2\varepsilon')F(e)m +1)\leq (1-2\varepsilon')m +\binom{L-1}{k-1}\leq (1-\varepsilon')m.$$
Note also that for every edge $e= \{V_{j_1},\dots,V_{j_k}\}\in E(R)$, we have that $(V_{j_1},\dots,V_{j_k})_{G'}$ is $\varepsilon'$-regular with density $d> \varepsilon'$. So indeed, by the definition of $\varepsilon'$-regularity, it is possible to successively add edges to $M$ in order to obtain a matching $M$ as desired.

Note that the size of $M$ is
$$\sum\limits_{e\in E(R)} K_e \geq \sum\limits_{e\in E(R)} (1-2\varepsilon')F(e)m=(1-2\varepsilon')m(b+\varepsilon)L\geq (1-2\varepsilon')(b+\varepsilon)(1-\varepsilon')n\geq bn.$$
So indeed $G$ has an (integer) matching of size at least $bn$.
\end{proof}
We are now in a position to prove Theorem~\ref{main theorem}.
\removelastskip\penalty55\medskip\noindent{\bf Proof of Theorem~\ref{main theorem}.}
Let $\varepsilon' >0$ and let $0<\varepsilon'' \ll \varepsilon', \varepsilon, 1/k, 1/2(k-d)-a$. Let $n_0\in \mathbb{N}$ be sufficiently large and suppose that $n\geq n_0$. Let $k':= k-d$ and $n':= n-d$. Then
\begin{align*}
f_0^{(a+\varepsilon'')n}(k-d,n-d)&\hspace{0.14cm}=f_0^{(a+\varepsilon'')(n'+d)}(k',n')\leq f_0^{(a+2\varepsilon'')n'}(k',n')\leq m_0^{(a+2\varepsilon'')n'}(k',n')\\
&\stackrel{(\ref{Frankl equation})}{\leq} \left( 1-(1-a-2\varepsilon'')^{k'}+\frac{\varepsilon'}{4}\right)\binom{n'}{k'} \leq \left( 1-(1-a)^{k-d}+\frac{\varepsilon'}{2}\right)\binom{n-d}{k-d}.
\end{align*}
So by Proposition~\ref{main proposition}, if $H$ is a $k$-uniform hypergraph on $n$ vertices such that for at least $(1-\varepsilon'')\binom{n}{d}$ $d$-tuples of vertices $L\in \binom{V(H)}{d}$ we have
$$deg_H(L)\geq \frac{ 1-(1-a)^{k-d}+\varepsilon'/2}{(k-d)!} n^{k-d}
\ge \left( 1-(1-a)^{k-d}+\frac{\varepsilon'}{2}\right)\binom{n-d}{k-d},$$
then $H$ has a fractional matching of size $(a+\varepsilon'')n$. So by Lemma~\ref{fractional to integer matchings}, any $k$-uniform hypergraph $G$ on $n \ge n_0'$ vertices (where $n_0'$ is sufficiently large) with
$$\delta_d(G)\geq \left( 1-(1-a)^{k-d}+\varepsilon' \right)\binom{n-d}{k-d}
 \geq \left(\frac{ 1-(1-a)^{k-d}+\varepsilon'/2}{(k-d)!}+ \varepsilon''\right) n^{k-d} $$
has an (integer) matching of size at least $an$. This gives the upper bound in Theorem~\ref{main theorem}. The lower bound follows from (\ref{fractional matching general lower bound}).
\endproof

We can prove Theorem \ref{minimum degree theorem actual perfect} in a similar way, but to do so we will need to use the absorbing technique as introduced by R\"{o}dl, Ruci\'nski and Szemer\'edi~\cite{RRSone}. More precisely, we use the existence of a small and powerful matching $M_{abs}$ in $G$ which, by `absorbing' vertices, can transform any almost perfect matching into a perfect matching. $M_{abs}$ has the property that whenever $X$ is a sufficiently small set of vertices of $G$ not covered by $M_{abs}$ (and $|X|\in k\mathbb{N}$) there exists a matching in $G$ which covers precisely the vertices in $X\cup V(M_{abs})$. Since this part of the proof of Theorem~\ref{minimum degree theorem actual perfect} is very similar to the corresponding part of the proof of Theorem 1.1 in \cite{Large matchings}, we only sketch it.

\removelastskip\penalty55\medskip\noindent{\bf Proof of Theorem~\ref{minimum degree theorem actual perfect} (sketch).}
Let $\varepsilon >0$ and suppose that $G$ is a $k$-uniform hypergraph on $n$ vertices with minimum $d$-degree at least\COMMENT{COMMENT: This inequality holds, since $d<k/2$ by assumption, so as $d,k$ are integers we have that $k/2-d\geq 1/2$. Hence, as $k\geq 3$,
$$\left(\frac{k-d}{k}-\frac{k-d-1}{k^{k-1}} \right)\geq \left(\frac{k/2+1/2}{k}-\frac{k-d-1}{k^{k-1}} \right)\geq 1/2.$$}
\begin{equation}\label{Absorbing equation}
\left(\frac{k-d}{k}-\frac{k-d-1}{k^{k-1}} +\varepsilon \right)\binom{n-d}{k-d}>\left(\frac{1}{2}+\varepsilon \right)\binom{n-d}{k-d}.
\end{equation}
(\ref{Absorbing equation}) implies that we can use the Strong Absorbing Lemma from \cite{Absorbing} to find an absorbing matching $M_{abs}$ in $G$, and set $G':=G\backslash V(M_{abs})$. Using the degree condition, Theorem \ref{minimum degree theorem perfect} gives us a perfect fractional matching in $G'$ for sufficiently large $n$. Lemma \ref{fractional to integer matchings}{\rm (ii)} then transforms this into an almost perfect integer matching $M_{alm}$ in $G'$. We then extend $M_{alm}\cup M_{abs}$ to a perfect matching of $G$ by using the absorbing property of $M_{abs}$.
\endproof

\section{Concluding remarks}\label{concluding remarks}
Using a similar method to that employed in proving Theorem \ref{minimum degree theorem actual perfect} it is possible to prove a variant of Theorem \ref{Frankl theorem} that verifies Conjecture \ref{Erdos main conjecture} asymptotically for all $s\in \mathbb{N}$ satisfying $s/n\leq a_k$, where $a_k$ is the unique solution in $(0,1/(k+1))$ to
$$1=\frac{1-(1-2a_k)^{k-1}}{(1-a_k)^{k-1}}.$$
For small values of $k$ this allows us to verify Conjecture \ref{Erdos main conjecture} asymptotically for some values of $s$ not covered by Theorem \ref{Frankl theorem}. For example for $k=4$ this allows $s$ to range up to $0.567n/k$. This approach also yields a slight improvement, for small values of $k$, to the range of matching sizes allowed in Theorem \ref{main theorem}, (at the expense of lengthy calculations). But for large $k$ Theorem~\ref{Frankl theorem} gives the better bounds on the matching sizes allowed (as $a_k$ is close to $0.48/k$ in this case). For details see \cite{Townsend PhD}.

\medskip

{\footnotesize \obeylines \parindent=0pt

School of Mathematics
University of Birmingham
Edgbaston
Birmingham
B15 2TT
UK
}
\begin{flushleft}
{\it{E-mail addresses}:}
{\rm{\{d.kuhn, d.osthus, txt238\}@bham.ac.uk}}
\end{flushleft}

\end{document}